\documentclass[12pt]{amsart}

\usepackage[margin=1.15in]{geometry}
\usepackage{amscd,amssymb, amsmath,amsfonts, wasysym, mathrsfs, enumitem, mathtools,hhline,color}
\usepackage{graphicx}
\usepackage[all, cmtip]{xy}

\usepackage{url}

\definecolor{hot}{RGB}{65,105,225}

\usepackage[pagebackref=true,colorlinks=true, linkcolor=hot ,  citecolor=hot, urlcolor=hot]{hyperref}

\theoremstyle{plain}
\newtheorem{theorem}{Theorem}[section]
\newtheorem{prop}[theorem]{Proposition}

\newtheorem{cor}[theorem]{Corollary}

\newtheorem{lemma}[theorem]{Lemma}
\newtheorem{thrm}[theorem]{Theorem}

\theoremstyle{definition}

\newtheorem{defn}[theorem]{Definition}

\newtheorem{rmk}[theorem]{Remark}

\newtheorem*{ex*}{Example}

\newcommand\sV{{\mathcal V}}

\newcommand\scrY{\mathscr{Y}}
\newcommand\scrX{\mathscr{X}}

\def\bD{\mathbb{D}}

\newcommand\kk{{\mathbb{K}}}

\newcommand\qq{{\mathbb{Q}}}
\newcommand\zz{{\mathbb{Z}}}
\newcommand\rr{{\mathbb{R}}}
\newcommand\cc{{\mathbb{C}}}

\newcommand\nn{{\mathbb{N}}}

\newcommand{\mb}{\mathcal{M}_{\textrm{B}}}

\DeclareMathOperator{\Supp}{Supp}                
                  
\DeclareMathOperator{\Exp}{Exp}
\DeclareMathOperator{\homo}{Hom}

\DeclareMathOperator{\spec}{Spec}

\DeclareMathOperator{\link}{Link}

\DeclareMathOperator{\Ann}{Ann}

\def\ra{\rightarrow}

\def\bone{\mathbf{1}}
\def\bC{\mathbb{C}}
\def\cM{\mathcal{M}}
\def\cV{\mathcal{V}}

\def\al{\alpha}

\def\bP{\mathbb{P}}

\def\pa{\partial}

\def\cD{\mathcal{D}}
\def\cO{\mathcal{O}}
\def\lra{\longrightarrow}
\def\bQ{\mathbb{Q}}

\def\cL{\mathcal{L}}

\def\bZ{\mathbb{Z}}

\def\lam{\lambda}
\def\sX{\mathscr{X}}

\newcommand{\ubul}{{\,\begin{picture}(-1,1)(-1,-3)\circle*{2}\end{picture}\ }}

\title{Local systems on analytic germ complements}
\begin{document}

\author{Nero Budur}
\address{KU Leuven, Department of Mathematics,
Celestijnenlaan 200B, B-3001 Leuven, Belgium} 
\email{Nero.Budur@wis.kuleuven.be}

\author{Botong Wang}
\address{KU Leuven, Department of Mathematics,
Celestijnenlaan 200B, B-3001 Leuven, Belgium} 
\email{botong.wang@wis.kuleuven.be}

\maketitle

\begin{abstract}
We prove that the cohomology jump loci of rank one local systems on the complement in a small ball of a germ of a complex analytic set are finite unions of torsion translates of subtori. This is a generalization of the classical Monodromy Theorem stating that the eigenvalues of the monodromy on the cohomology of the Milnor fiber of a germ of a holomorphic function are roots of unity.
\end{abstract}

\section{Introduction}

Let $(\mathscr{X},0)\subset (\bC^n,0)$ be the germ of a complex analytic set. Let $B$ be a small open ball at the origin in $\bC^n$, and $U=B\setminus \sX$. Up to homotopy equivalence, the open set $U$  is uniquely determined by the germ $\sX$. We call such open set $U$ the small ball complement of $\sX$. 

By the local conic structure of analytic sets, $U$ has the same homotopy type as the link complement $\partial (\bar{B})\setminus \cL$, where the link of $(\sX,0)$ is defined as $\cL=\partial (\bar{B})\cap \sX$, and $\pa (\bar{B})$ is the boundary sphere of $B$. We will however keep working throughout this paper with small ball complements of analytic germs rather than link complements.

Let $\mb(U)$ be the moduli space of rank 1 local systems on $U$. By identifying a local system with the monodromies around loops in $U$, one can write $\mb(U)\cong\homo(H_1(U,\bZ),\bC^*)$. The cohomology jump loci of  $U$ are defined as usual by 
$$
\sV^i_k(U)=\left\{L\in \mb(U)|\dim H^i(U, L)\geq k\right\}.
$$
Cohomology of local systems can be computed from twisted cochain complexes on the universal cover. It is known that this implies that $\sV^i_k(U)$ are affine $\bZ$-schemes of finite type for any topological space $U$ that has the homotopy type of a finite CW-complex, and that $\sV^i_k(U)$ depend only on the homotopy type of $U$.

The main result is the following: 
\begin{theorem}\label{thrmMain0} Let $U$ be the small ball complement of the germ of a complex analytic set. Then each irreducible component of $\sV^i_k(U)$ is a torsion translate of subtorus of $\mb(U)$. 
\end{theorem}

Here, by a subtorus of $\mb(U)$, we mean an affine algebraic subtorus $(\bC^*)^{p}$  of the identity-containing component $\mb(U)_\bone$ of $\mb(U)$. In particular, all such subtori are defined over ${\bZ}$. Note that $\mb(U)_\bone\cong (\bC^*)^{r}$, where $r$ is the first Betti number of $U$.

The homotopy type of the small ball complement of a complex singularity is much more restricted than that of the link of the singularity. Indeed,  by a result of Kapovich-Koll\'ar \cite{KK}, every finitely presented group can be achieved as the fundamental group of the link of an isolated complex singularity. Hence the cohomology jump loci $\sV^1_1$ of such links can be any affine $\bZ$-schemes of finite type, since it is known how $\sV^1_1$ can be defined entirely from $\pi_1$ for any topological space.

In the last section of this paper, we also consider small ball complements in singular ambient spaces.

Theorem \ref{thrmMain0} is the local counterpart of a similar statement proven for the cohomology jump loci of smooth complex quasi-projective varieties in \cite{BW} and of compact K\"ahler manifolds in \cite{W16}. These results built on a long series of partial results due to Green-Lazarsfeld, Arapura, Simpson, Dimca-Papadima, etc. For precise references see the recent survey \cite{BW-surv}. An interesting feature of this paper is that it does not rely on any Hodge theory, formality properties, nor on previously proven cases.

An easy consequence of the structure result for cohomology jump loci is invariance under taking the inverse, which is also the dual, of a rank one local system:

\begin{thrm}\label{corInv}
Let $X$ be a smooth complex quasi-projective variety, a compact K\"ahler manifold, or a small ball complement of the germ of a complex analytic set. Or, more generally, let $X$ be a topological space of homotopy type of a finite CW-complex such that the irreducible components of $\cV^i_k(X)$ are torsion translated subtori of $\mb(X)$. Then for any rank one local system $L$ on $X$, non-canonically, 
$$
H^i(X,L)\cong H^i(X,L^{-1}).
$$
\end{thrm}
The proof of this theorem is essentially the same as the proof by A. Dimca of \cite[Corollary 4.9]{D-ad}, where the case of $H^1$ for rank one local systems on smooth complex quasi-projective varieties was proved.

Theorem \ref{thrmMain0} is a generalization of the classical Monodromy Theorem stating that the eigenvalues of the monodromy on the Milnor fiber of a germ of a holomorphic function are roots of unity. More precisely, letting
$$
\cV(U)=\bigcup_{i\in\bZ} \cV^i_1(U)=\{L \in\mb(U) \mid H^\ubul (U,L)\ne 0\}
$$
be the cohomology support loci, we have:

\begin{prop}\label{propEig}
Let $f: (\cc^n, 0)\to (\cc, 0)$ be a germ of  a holomorphic function. Let $U$ be the small ball complement of $f^{-1}(0)$ and 
$$
f^*: \mb(\bC^*)=\bC^*\lra \mb(U) 
$$
be the inverse image map on local systems under $f|_U:U\ra \bC^*$. Then $(f^*)^{-1}(\cV(U))$ is the set of eigenvalues of the monodromy on the cohomology of the Milnor fiber of $f$ at $0$.
\end{prop}

This proposition is not new and it can be deduced from a more precise result, see \cite[Corollary 6.4.9]{Di}, or from the results in \cite{B}, as we will explain. It has also recently appeared as \cite[Proposition 6.6]{LM}.

A different attempt to generalize statements about the Milnor monodromy has been introduced by C. Sabbah \cite{Sab}. Given a collection $F=(f_1,\ldots,f_r)$ of germs of holomorphic functions $f_j:(\bC^n,0)\ra (\bC,0)$, he has defined a complex $\psi_F\bC$ with $A$-constructible cohomology on $\cap_j f^{-1}_j(0)$, where $A=\bC[t_1^{\pm},\ldots,t_r^{\pm}]$ is the affine coordinate ring of $(\bC^*)^r$. This complex is the analog of Deligne's nearby cycles complex $\psi_f\bC$ for the case $r=1$. While $\psi_f\bC$ governs the Milnor monodromy information, Sabbah's $\psi_F\bC$ governs  the more general Alexander-type invariants, see {\it loc. cit}. One of the main results of \cite{Sab} is about the support in $(\bC^*)^r$ of the stalk of $\psi_F\bC$ at the origin given by the $A$-module structure, denoted 
$$
\Supp_0(\psi_F\bC).
$$
$
\Supp_0(\psi_F\bC)
$ is shown to included in a hypersurface whose irreducible components are translated subtori. Nicaise \cite{Ni} showed further that the irreducible components of this hypersurface are translated by torsion points. As a corollary of Theorem \ref{thrmMain0} and the results in \cite{B}, one has the more general statement \footnote{This was stated as  Theorem 2 in \cite{B}, but the proof depended on \cite{Li-gap}.  \cite{Li-gap} turned out to have serious flaws in the proofs, see the comments below.}:

\begin{thrm}\label{thrmSab}
Let $F=(f_1,\ldots,f_r)$ be a collection of germs of holomorphic functions $f_i:(\bC^n,0)\ra (\bC,0)$. Then every irreducible component of $\Supp_0(\psi_F\bC)$ is a torsion translated subtorus of $(\bC^*)^r$.
\end{thrm}

In a subsequent article \cite{BLSW} with Y. Liu and L. Saumell, we show furthermore  that $\Supp_0(\psi_F\bC)$ is a hypersurface. 

The proof of Theorem \ref{thrmMain0} is reduced to the case of an analytic hypersurface germ by a result of J. Fern\'andez de Bobadilla \cite{FdB2}, see Theorem \ref{thrmBo}. Let $f: (\cc^n, 0)\to (\cc, 0)$ be a germ of  a holomorphic function and $\sX=f^{-1}(0)$. We call $U=B\setminus\sX$ the small ball complement of $f$. Since we will be only interested in the topology of $U$, we will assume $f$ has no irreducible factor with multiplicity higher than one. In other words, we assume that $f=\prod_{i=1}^rf_i$ where $f_i$ are distinct (up to multiplication by invertible holomorphic function germs) irreducible germs of holomorphic functions. Then $H_1(U, \zz)\cong \zz^r$ is generated by small loops around $\{f_i=0\}$ at a general point. Thus, 
$$\mb(U)\cong H^1(U, \cc^*)\cong (\cc^*)^r.$$
Given any $\lam\in (\cc^*)^r\in \mb(U)$, denote the corresponding local system by $L_\lam$.

The case of Theorem \ref{thrmMain0} for the small ball complement of the germ of a holomorphic function is stated erroneously as a theorem in \cite{Li-gap}. See \cite[Section 8]{BW-surv} for a discussion about the flawed argument appearing in the proof of the crucial \cite[Proposition 3.5]{Li-gap}. The proof in \cite{Li-gap} is unrepairable in our opinion.  A simple test that reveals major subtleties of the problem, and which \cite{Li-gap} fails, is the following: how does one show that an isolated point of $\cV^i_k(U)$ must be unitary, or, stronger, torsion? Isolated points are well-known to occur. $L_2$-techniques and Hodge theory can be employed to show that a unitary isolated point is torsion, but these tools, at least in the form quoted in \cite{Li-gap}, are not available without the unitarity assumption.


We sketch now the main steps of the proof of  the hypersurface case of Theorem \ref{thrmMain0}:

\begin{enumerate}[label=\it{Step \arabic*:}]
\item \label{step1} We reduce  the problem to the case of all $f_i$ being  polynomials defined over $\bar{\qq}$. This step uses a result due to S. Izumi \cite{I} based on the Artin Approximation Theorem, and a new result of M. Bilski, A. Parusi\'nski, and G. Rond \cite{BPR} generalizing a theorem of J. Fern\'andez de Bobadilla \cite{FdB}.

\item Using  the $\cD$-module $\mathcal{D}[s]f^s$, where $s=(s_1, \ldots, s_r)$ and $f^s=f_1^{s_1}\cdots f_r^{s_r}$, we define set-theoretically some cohomology jump loci $\Sigma^i_k(f)$ as subsets of $\cc^r$. We also prove $\Sigma^i_k(f)$ is preserved by the action of $Gal(\cc/\bar{\qq})$. 
\item \label{step3} Given any point $\lambda \in \mb(U)$, we prove that there exists a small open ball $N\subset \cc^r$  biholomorphic with an open neighborhood $\Exp(N)\subset (\cc^*)^r$ of  $\lam$,
and such that $\Exp(N\cap \Sigma^i_k(f))=\Exp(N)\cap \sV^i_k(U)$.  This step uses the Riemann-Hilbert correspondence between regular holonomic $\cD$-modules and perverse sheaves.
\item Using Step 3 and the fact that both $\Sigma^i_k(f)$ and $\sV^i_k(U)$ are preserved by $Gal(\cc/\bar{\qq})$, we prove that each irreducible component of $\sV^i_k(U)$ is a torsion translate of subtorus. This is result, Theorem \ref{goodexpo}, might be of independent interest. It is a generalization of a classical theorem of Gelfond and Schneider, which says that if $\alpha$ and $e^{2\pi i\alpha}$ are both algebraic numbers, then $\alpha\in \qq$. Theorem \ref{goodexpo} is proved using  a strong theorem of M. Laurent \cite{mm}.
\end{enumerate}

Section 2 deals with the first step. Section 3 contains the second and third step. Section 4 is devoted to the last step. Section 5 finishes the proofs for all the results in this introduction. Section 6 considers the case of  small ball complements in a singular ambient space.

The word germ in this paper will refer to complex analytic germs only.



{\bf Acknowledgement.}  We would like to thank H. Hauser, Y. Liu, M. Saito, and U. Walther for answering some of our questions, and to A. Dimca for pointing out the references in connection with Theorem \ref{corInv} and Proposition \ref{propEig}. The first author thanks S. Zucker and the Department of Mathematics of the Johns Hopkins University for the hospitality during the writing of this article.  The first author was partly sponsored by FWO, a KU Leuven OT grant, and a Flemish Methusalem grant.

\section{Deformations of  hypersurface singularities}\label{Def}

The goal of this section is to prove:

\begin{theorem}\label{Qbar2}
Let $f: (\cc^n, 0)\to (\cc, 0)$ be the germ of a holomorphic function defining the germ of a reduced analytic hypersurface. Let $r$ be the number of analytic branches of $f$. Then there exist irreducible polynomial functions $h_i:\cc^n\to \cc$ $(i=1,\ldots, r)$ defined over $\bar{\qq}$ such that, letting $h=\prod_{i=1}^rh_i$, we have:

\begin{enumerate}

\item[(1)] The germ of $h$ at $0$ defines a reduced analytic hypersurface germ, and $h=\prod_{i=1}^rh_i$ is  the irreducible decomposition (up to multiplication by invertible holomorphic germs) of the germ of $h$ at $0$.

\item[(2)] If $B$ is a small enough ball centered at $0$ in $\bC^n$, then $B\setminus f^{-1}(0)$ is homotopy equivalent to $B\setminus h^{-1}(0)$.

\item[(3)] If we set $W=\bC^n\setminus h^{-1}(0)$, the map
$$
\mb(W)\lra \mb(B\setminus h^{-1}(0))
$$
induced by the restriction map is an isomorphism. 
\end{enumerate}
\end{theorem}



First, we deform $f$ to a polynomial defined over $\cc$. This follows from a beautiful result of Bilski-Parusi\'nski-Rond \cite{BPR}, which generalizes a theorem of J. Fern\'andez de Bobadilla \cite{FdB}:

\begin{theorem}\label{thrmGerm}\cite[Theorem 1.2]{BPR}
Let $\kk=\cc$ or $\rr$. Let $f: (\kk^n, 0)\to (\kk, 0)$ be an analytic function germ. Then there is a homeomorphism $\sigma: (\kk^n, 0)\to (\kk^n, 0)$ such that $f\circ \sigma$ is the germ of a polynomial. 
\end{theorem}

Indeed, this theorem implies the following corollary. 

\begin{cor}\label{polynomial}
Let $f: (\cc^n, 0)\to (\cc, 0)$ be a holomorphic function germ. There exists a polynomial function germ $g: (\cc^n,0)\to (\cc, 0)$ such that the small ball complement of $f$ has the same homotopy type as the small ball complement of $g$. 
\end{cor}
\begin{proof}
Note that small ball complements of analytic hypersurfaces are defined uniquely only up to homotopy type. Let $\sigma$ be the local homeomorphism from Theorem \ref{thrmGerm} and let $g=f\circ \sigma$. Let $B\subset \bC^n$ be a small enough ball centered at $0$. Note that $\sigma (B)$ is not necessarily a ball anymore. In any case, we can use the theory of good neighborhoods of D. Prill \cite{pr}. A good neighborhood of a point $y$ in a topological space $X$ with respect to a subspace $Y$ is an open neighborhood $U$ of $y$ in $X$ such that: there exists a collection $\{U_a\}_{a\in A}$ as a basis of neighborhoods of $y$, with $U_a\setminus Y$ a deformation retract of $U\setminus Y$ for all $a\in A$. It is proven in {\it loc. cit.} that:
\begin{itemize}
\item $U_a$ are also good neighborhoods of $y$ in $X$ with respect to $Y$;
\medskip
\item the homotopy type of $U\setminus Y$ depends only on the triple $\{X,Y,y\}$;
\medskip
\item good neighborhoods exist for triples $\{X,Y, y\}$ where $X$ is an analytic variety and $Y$ a subvariety.
\end{itemize}
In particular, $B$ is a good neighborhood of $0$ in $\bC^n$ with respect to $g^{-1}(0)$ if $B$ is small enough. Since $\sigma$ is a local homeomorphism, $\sigma(B)$ is also a good neighborhood of $0$ in $\bC^n$ with respect to $f^{-1}(0)$ since the conditions are satisfied for $\{\sigma (U_a)\}_{a\in A}$. In particular, $\sigma(B\setminus g^{-1}(0))$ has the homotopy type of a small ball complement of $f^{-1}(0)$.
\end{proof}

\begin{prop}\label{irreduciblepolynomial}
Given a polynomial $g:\cc^n\to \cc$, let $r$ be the number of analytic branches of the reduced germ underlying  the germ of $g$ at $0$. Then there exist irreducible polynomials $\hat{h}_i: \cc^n\to \cc$ $(i=1,\ldots, r)$ such that, letting $\hat{h}=\prod_{i=1}^r\hat{h}_i$, we have:
\begin{itemize}

\item[(1)] The germ of $\hat{h}$ at $0$ defines a reduced analytic hypersurface, and $\hat{h}=\prod_{i=1}^r\hat{h}_i$ is  the irreducible decomposition (up to multiplication by invertible holomorphic germs) of the germ of $\hat{h}$ at $0$.

\item[(2)] If $B$ is a small enough ball centered at $0$ in $\bC^n$, then $B\setminus g^{-1}(0)$ is homotopy equivalent to $B\setminus \hat{h}^{-1}(0)$.
\end{itemize}

\end{prop}
\begin{proof}

Note that once everything else  is achieved, the condition that $\hat{h}_i$ are irreducible is easy to satisfy by throwing away the polynomial factors not passing through the origin. Hence we will not worry about this condition from now on.

The first step was suggested to us by Morihiko Saito. Let $\prod_{i=1}^{r'}g_i$ be an irreducible decomposition of the germ of $g$ at $0$. Note that $r'\ge r$, with equality if and only if the germ of $g$ at $0$ is reduced. Relabel the $g_i$ such that the first $r$ of them define the $r$ analytic branches of the reduced analytic germ underlying the germ of $g$. Thus $g$ and $\prod_{i=1}^rg_i$ have the same small ball complement at the origin.

Instead of holomorphic function germs, we will consider $g_i$ as converging power series. Recall that an algebraic power series $\phi$ is a formal power series satisfying some non-trivial polynomial equation $P(\phi)=0$, where $P(T)\in \bC[x_1,\ldots,x_n][T]$. The ring of algebraic power series, which we denote by $\bC\langle x_1,\ldots x_n \rangle$, can also be described as the local ring at $0$  of regular functions in the \'etale topology of $X=\bC^n=Spec\; \bC[x_1,\ldots,x_n]$, or in other words,
$$
\bC\langle x_1,\ldots x_n \rangle =\varinjlim \Gamma (V,\cO_V),
$$
where the limit is over  \'etale neighborhoods $V\ra X$ of $0$. See \cite[I.4]{M} for details. 

Local analytic branches are \'etale locally algebraic. More precisely, by Artin Approximation Theorem \cite{A} and a corollary of it pointed out by Izumi \cite[Theorem E]{I}, see also \cite[Corollary 4.0.16]{HR}, if $g\in\bC[x_1,\ldots, x_n]$ splits as a product of formal power series $g=\prod_{i=1}^{r'} g_i$ with $g_i\in\bC[[x_1,\ldots,x_n]]$, then there exists a splitting $g=\prod_{i=1}^{r'}\hat{g}_i$ with $\hat{g}_i$ being algebraic power series and such that $\hat{g}_i=u_i g_i$ for some invertible formal power series $u_i$. Here we recall that the local ring $\bC\langle x_1,\ldots x_n \rangle$ is excellent Henselian, and hence the theorem of Izumi applies.



We apply next the following to the germs of $g_i$ at the origin: 

\begin{thrm} {\rm (\cite[Theorem 5.4]{BPR})}\label{algseries}
Let $\hat{g}_i: (\bC^n,0)\ra (\bC,0)$ be a finite family of algebraic power series. Then there exists a diffeomorphism $\sigma:(\bC^n,0)\ra (\bC^n,0)$ and analytic units $v_i: (\bC^n,0)\ra (\bC,0)$, $v_i(0)\ne 0$, such that for all $i$, $v_i(x)\hat{g}_i(\sigma(x))$ are germs of polynomials.
\end{thrm}

Let $\hat{g}=\prod^r_{i=1}\hat{g}_i$. Since $\hat{g}$ and $\prod_{i=1}^rg_i$ only differ by a unit, they have the same zero locus. Hence $\hat{g}$ has the same small ball complement at $g$. Let $\hat{h}_i=v_i(x)\hat{g}_i(\sigma(x))$ and $\hat{h}=\prod^r_{i=1}\hat{h}_i$. Then $\hat{h}$ and $\hat{g}$ differ by a analytic unit and a diffeomorphism $\sigma:(\bC^n,0)\ra (\bC^n,0)$. Again, the analytic unit does not change the small ball complement. Thus, the second part of the proposition follows from the same argument of the proof of Corollary \ref{polynomial}. 

Since $H^1(B\setminus \hat{h}^{-1}(0), \zz)\cong H^1(B\setminus g^{-1}(0), \zz)\cong \zz^r$, the irreducible decomposition of $\hat{h}$ as holomorphic function germ has $r$ factors. Recall that $\hat{h}=\prod^r_{i=1}\hat{h}_i$, where $\hat{h}_i$ are distinct polynomial function germs vanishing at the origin. Therefore, each polynomial function germ $\hat{h}_i$ is also irreducible as a holomorphic function germ. Thus the first part of the proposition follows.



\end{proof}

Next, we will deform the polynomials $\hat{h}_i$ to polynomials $h_i$ defined over $\bar{\qq}$. We use a similar idea to one in \cite{S} and \cite{BW}. 

\begin{prop}\label{Qbarpolynomial}
Let $\hat{h}_i: \cc^n\to \cc$ $(i=1,\ldots ,r)$ be irreducible polynomial functions. Let $\hat{h}=\prod_{i=1}^r\hat{h}_i$. Assume that the analytic germ of $\hat{h}$ at $0$ is reduced and the $\hat{h}_i$  define the mutually distinct analytic branches.  Then there exist irreducible polynomial functions $h_i:\cc^n\to \cc$ $(i=1,\ldots, r)$ defined over $\bar{\qq}$ such that, letting $h=\prod_{i=1}^rh_i$, we have:

\begin{enumerate}

\item[(1)] The germ of $h$ at $0$ defines a reduced analytic hypersurface germ, and $h=\prod_{i=1}^rh_i$ is  the irreducible decomposition (up to multiplication by invertible holomorphic germs) of the germ of $h$ at $0$.

\item[(2)] If $B$ is a small enough ball centered at $0$ in $\bC^n$, then $B\setminus \hat{h}^{-1}(0)$ is homotopy equivalent to $B\setminus h^{-1}(0)$.
\end{enumerate}
\end{prop}
\begin{proof}
First, we consider $\hat{h}_i$ as a global polynomial functions defined on $\cc^n$. Denote the zero locus of $\hat{h}_i$ in $\cc^n$ by $D_i$. Let $D$ be the zero locus of $\hat{h}$. Let 
\begin{equation}\label{resolution}
X_{l}\xrightarrow{\pi_{l}}X_{l-1}\xrightarrow{\pi_{l-1}}\cdots \xrightarrow{\pi_{1}}X_{0}=\cc^n
\end{equation}
be a log resolution of the pair $(\cc^n, D)$. Here, we assume each $\pi_j$ is the blow-up along a smooth center. Denote the composition $\pi_1\circ\cdots\circ \pi_l$ by $\pi$. Moreover, we can assume that $E{=}\pi^{-1}(0)$ is a union of simple normal crossing divisors in $X_l$. 

Let $R$ be a finitely generated subring (over $\qq$) of $\cc$ such that all $D_i$ and all $X_j$ together with all morphisms $\pi_j$ are defined over $R$. In other words, the polynomials $\hat{h}_i$ and all the centers of the blow-ups are defined over $R$. In particular, $\hat{h}_i\in R[x_1, \ldots, x_n]$, where $x_1, \ldots, x_n$ are the coordinates of $\cc^n$. 

We can consider $\spec(R)$ as a variety. The natural inclusion $\iota: R\to \cc$ corresponds to a $\cc$-point of $\spec(R)$, which we also denote by $\iota$. Since the sequence of morphisms (\ref{resolution}) is defined over $R$, we can consider it as a sequence of morphisms over $\spec(R)$. To emphasize the sequence is over $\spec(R)$, we use an extra subscript $R$ as in the following: 
\begin{equation}\label{Rresolution}
X_{l, R}\xrightarrow{\pi_{l, R}}X_{l-1, R}\xrightarrow{\pi_{l-1, R}}\cdots \xrightarrow{\pi_{1, R}}X_{0, R}=\spec(R)^{\times n}.
\end{equation}
The restriction of (\ref{Rresolution}) to $\iota\in \spec(R)$ is isomorphic to (\ref{resolution}). Given a $\cc$-point $\sigma\in \spec(R)$, we denote the restriction of the sequence (\ref{Rresolution}) to $\sigma$ by
\begin{equation}\label{sigmaresolution}
X_{l, \sigma}\xrightarrow{\pi_{l, \sigma}}X_{l-1, \sigma}\xrightarrow{\pi_{l-1, \sigma}}\cdots \xrightarrow{\pi_{1, \sigma}}X_{0, \sigma}=\cc^n.
\end{equation}
Denote the composition $\pi_{1, \sigma}\circ\cdots\circ\pi_{l, \sigma}$ by $\pi_{\sigma}$. When $\sigma\in \spec(R)$ is a general point, (\ref{sigmaresolution}) is also a resolution of singularity such that $E_{\sigma}\stackrel{\textrm{def}}{=}\pi_{\sigma}^{-1}(0)$ is a normal crossing divisor. By the proof of \cite[Theorem 2.1]{BW}, when $\sigma$ is a general $\cc$-point in $\spec(R)$ close to $\iota$, the pair $(X_{l, \sigma}, \pi^{-1}_{\sigma}(D_\sigma))$ is homeomorphic to the pair $(X_l, \pi^{-1}(D))$ in the stratified sense, and hence the pair $(X_{l, \sigma}, E_\sigma)$ is homeomorphic to the pair $(X_l, E)$ in the stratified sense. Clearly the $\bar{\qq}$-points are Zariski dense in $\spec(R)$. Thus we can choose a $\bar{\qq}$-point $\sigma$ such that $(X_{l, \sigma}, \pi^{-1}(D_\sigma))$ is homeomorphic to $(X_l, \pi^{-1}(D))$ in the stratified sense. Fix such $\bar{\qq}$-point $\sigma$, or equivalently a ring map $\sigma: R\to \bar{\qq}$. The ring map $\sigma$ extends to $\sigma: R[x_1, \ldots, x_n]\to \bar{\qq}[x_1, \ldots, x_n]$. Recall that $\hat{h}_i\in R[x_1, \ldots, x_n]$. Thus we can define $h=\sigma(\hat{h})$ and $h_i=\sigma(\hat{h}_i)$ as polynomials defined over $\bar{\qq}$. 

Let $T(E)$ be a tubular neighborhood of $E$ in $X_l$, and let $T(E_{\sigma})$ be a tubular neighborhood of $E_{\sigma}$ in $X_{l, \sigma}$. Since the pair $(X_{l, \sigma}, \pi^{-1}(D_\sigma))$ is homeomorphic to the pair $(X_l, \pi^{-1}(D))$ in the stratified sense, the complement of $E$ in $T(E)$  is homeomorphic to the complement of $E_\sigma$ in $T(E_\sigma)$. Notice that $\pi$ induces a homeomorphism between the complement of $E$ in $T(E)$ and a small ball complement of $\hat{h}$. The same is true for the complement of $E_\sigma$ in $T(E_\sigma)$ and a small ball complement of $h$. Therefore, a small ball complement of $\hat{h}$ is homeomorphic to a small ball complement of $h$. Since $\hat{h}_i$ are irreducible polynomials, and since $\sigma$ is close to $\iota$, $h_i$ are also irreducible. Therefore, the second part of the proposition follows.

Next, we need an elementary fact about the irreducibility of an analytic germ. 
\begin{lemma}
An analytic germ of a hypersurface $Z\subset X$ ($X$ smooth) at $P$ is irreducible if and only if for any  log resolution of $Z$, $\pi: Y\to X$, the preimage of $P$ in $\tilde{Z}$, i.e. $\tilde{Z}\cap \pi^{-1}(P)$, is connected. Here, $\tilde{Z}$ is the strict transform of $Z$. 
\end{lemma}
\begin{proof}
Suppose $Z$ is irreducible. Then the normalization of $Z$, denoted by $Z_{norm}$, is the germ of a normal analytic variety (see \cite[Theorem 6, page 203]{G}). In other words, the preimage of $P$ under the normalization map $Z_{norm}\to Z$ consists of only one point. Since $\tilde{Z}$ is normal, the map $\tilde{Z}\to Z$ factors through $\tilde{Z}\to Z_{norm}$. Since any bimeromorphic finite morphism to a normal variety is always an isomorphism, by Zariski's main theorem, $\tilde{Z}\to Z_{norm}$ has connected fibers. Therefore, the preimage of $P$ in $\tilde{Z}$ is connected. 

Conversely, suppose $Z$ is reducible. Let $B$ be a sufficiently small ball in $X$ centered at $P$. Then $Z\cap B$ is reducible, and hence $\tilde{Z}\cap \pi^{-1}(B)$ is also reducible. Since $\tilde{Z}$ is smooth, $\tilde{Z}\cap \pi^{-1}(B)$ is not connected. Since $\pi$ is a holomorphic map and $B$ is sufficiently small, $\tilde{Z}\cap \pi^{-1}(P)$ is not connected. 
\end{proof}

The pair $(X_{l, \sigma}, \pi^{-1}(D_\sigma))$ is homeomorphic to the pair $(X_l, \pi^{-1}(D))$ in the stratified sense. Thus by the preceding lemma, $h_i$ is irreducible as a holomorphic function germ since so is $\hat{h}_i$.  

The reducedness of the germ at $0$ defined by $\hat{h}=\prod_{i=1}^r\hat{h}_i$ follows as before from the fact $r$ must be the number of analytic branches, due to the homotopy equivalence between the small ball complements for $\hat{h}$ and $h$.

\end{proof}

\begin{proof}[Proof of Theorem \ref{Qbar2}]
Let $h$ and $h_i$ be defined as in Proposition \ref{Qbarpolynomial}. The first two parts of the theorem follows from Corollary \ref{polynomial}, Proposition \ref{irreduciblepolynomial}, Theorem \ref{algseries} and  Proposition \ref{Qbarpolynomial}. 

To prove the third part, let $H_i$ be the closure of $h_i^{-1}(0)$ in $\bP^n$. $H_i$ are irreducible of degree say $d_i$. Let $H_0$ be the hyperplane at infinity. Then $W=\bP^n\setminus (\cup_{i=0}^r H_i).$ By a well-known computation,
$$
H_1(W,\bZ)\cong \left (\bigoplus_{i=0}^r\bZ \cdot[H_i]\right) \bigg/ \left (\sum_{i=0}^rd_i[H_i]\right).
$$
The isomorphism can be made canonical  by fixing the meaning of the classes $[H_i]$. Since $d_0=1$, it follows that $H_1(W,\bZ)\cong \bZ^r$. In particular, there are canonical isomorphisms $\mb(W)\cong (\bC^*)^r \cong \mb (B\setminus h^{-1}(0))$ given by identifying local systems of rank one with their monodromies around the divisors. This finishes the proof of Theorem \ref{Qbar2}.
\end{proof}


\section{$\cD$-modules}\label{dmod}

By the previous section, we can restrict ourselves to the following setup. Let $f_i\in\cc[x_1,\ldots,x_n]$ for $1\le i\le r$ be distinct (up to a multiplication by a nonzero constant) irreducible polynomials, defined over $\bar{\qq}$. Set $f=\prod_{i=1}^rf_i$, $X=\bC^n$, $D=f^{-1}(0)$, $W=X\setminus D$.  Denote the origin of $\cc^n$ by $O$, and let $u:\{O\}\ra X$ denote the natural inclusion. Let $B$ a very small open ball in $X$ centered at $O$, and let $U=B\setminus (B\cap D)$. We assume in this section that the restriction
$$
\mb(W)\xrightarrow{\sim}\mb(U)
$$
is an isomorphism, and that
$
\mb(U)
$
is identified with $(\bC^*)^r$ by monodromies around general points around the divisors given by $f_i$.

Let $\cD_X$ be the sheaf of algebraic linear differential operators on $X$. We introduce a multivariable ${s}=(s_1, \ldots, s_r)$ and the product notation
$$
{f}^{{s}}=\prod_{i=1}^rf_i^{s_i}.
$$
Let $$\cD_X[s]f^s=\cD_X[s_1, \ldots, s_r]f_1^{s_1}\cdots f_r^{s_r}$$ be the natural left $\cD_X[s]$-module. Here, for $P\in\cD_X[s]$, the expression $Pf^s$ in $\cD_X[s]f^s$ means the result of the application of the operator $P$ on $f^s$ as expected.  As left $\cD_X[s]$-modules, $\cD_X[s]f^s\cong \cD_X[s]/\Ann(f^s)$, where $\Ann(f^s)$ denotes the left ideal  of operators $P\in\cD_X[s]$ with $Pf^s=0$. In other words, $\cD_X[s]f^s$ is a $\cD_X[s]$-submodule of the rank one free $\cO_X[f^{-1},s]$-module generated by  $f^s$.

Let $Mod(\cD_X)$ denote the category of left $\cD_X$-modules. Recall that there is an ``easy" pullback functor $$u^*: Mod(\cD_X)\to Mod(\cD_O)$$ is defined by $$M\mapsto \cD_{O\to X}\mathop{\otimes}_{u^{-1}(\cD_X)}u^{-1}(M),$$
where $u^{-1}$ is the sheaf theoretic pullback. In this case, the transfer $(\cD_{O},u^{-1}(\cD_X))$-bimodule $\cD_{O\to X}$ is isomorphic to
$$\bC[x_1,\ldots,x_n, \partial/\partial x_1, \ldots, \partial/\partial x_n]/(x_1,\ldots,x_n)\cong \bC[\partial/\partial x_1, \ldots, \partial/\partial x_n].$$ 
Here the right $u^{-1}(\cD_X)$-module structure on $\bC[\partial/\partial x_1, \ldots, \partial/\partial x_n]$ is given by the following. The multiplication by $\partial/\partial x_i$ is simply the usual multiplication and the multiplication by $x_i$ is given by 
$$\left((\partial/\partial x_1)^{m_1}\cdots(\partial/\partial x_n)^{m_n}\right)\cdot x_i=m_i(\partial/\partial x_1)^{m_1}\cdots(\partial/\partial x_i)^{m_i-1}\cdots(\partial/\partial x_n)^{m_n}.$$
By taking the derived tensor product, one has the left derived functor of $u^*$,
$$
Lu^*:D^b(\cD_X)\to D^b(\cD_O),
$$
on the bounded derived category of left $\cD$-modules. The special $\cD$-module inverse image 
$$u^+: D^b(\cD)\to D^b(\cD_O)$$
is defined as
$$
u^+=Lu^*[-n].
$$
Notice that a $\cD_O$-module is nothing but a $\bC$ vector space, and so $D^b(\cD_O)$ is the bounded derived category of complexes of vector spaces.

The left adjoint of $u^+$ is the functor
$$
u^\bigstar = \bD_O \circ u^+ \circ\bD_X : D^b(\cD_X)\to D^b(\cD_O).
$$
Here $\bD_X$ is the dualizing functor,
$$
\bD_X=R\mathcal{H}om_{\cD_X}(\;.\;,\cD_X)\mathop{\otimes}_{\cO_X}\omega_X^{\otimes -1}[n] : D^b(\cD_X) \to D^b(\cD_X),
$$
where $\omega_X= \cO_Xdx_1\wedge\ldots\wedge dx_n$ is the canonical sheaf on $X$.



\begin{defn}
We define set-theoretically the {\it $\cD$-module cohomology jump loci of $f$ at the origin} by
$$\Sigma^i_k(f,O)=\left\{\al\in \cc^r\mid\dim_{\cc} H^i\left(u^{\bigstar}\left(\cD_X[s]f^s\otimes_{\bC[s]}\bC[s]/(s-\al)\right)\right)\geq k\right\}.$$
Here, by $\cc[s]/(s-\al)$ we mean the quotient of $\bC[s_1,\ldots,s_r]$ by the ideal $(s_1-\al_1,\ldots,s_r-\al_r).$ To simplify the notation when the context is clear, we will also denote this set by $\Sigma^i_k(f)$.
\end{defn}

\begin{rmk}
It will be clear from the proof of Theorem \ref{thrmS} below that the sets $\Sigma^i_k(f)$ typically contain infinitely many irreducible components and that the $\bZ^r$ translates of each component appear. Hence $\Sigma^i_k(f)$ are not Zariski closed in $\bC^r$. In particular, the method employed in \cite{BW} for showing the torsion translated subtori property of $\cV^i_k(U)$ for quasi-projective $U$ is not immediately available here. However, by restricting to smaller analytic open subsets of $\bC^r$, one only sees finitely many components in $\Sigma^i_k(f)$. In this article, the main idea  to adapt the method of \cite{BW} to work with this type of finiteness only.
\end{rmk}

Define the analytic map
$
\Exp : \bC^r\ra (\bC^*)^r
$
by  $\al\mapsto\exp(2\pi i \al)$. The goal in this section is to prove:

\begin{thrm}\label{thrmS} With the assumptions on $f$ from the beginning of this section, we have:
\begin{enumerate}
\item $\Sigma^i_k(f)$ is preserved under the action of $Gal(\cc/\bar{\qq})$.
\item  Let $\lam\in (\bC^*)^r$. Then there exists an open ball  $N\subset\bC^r$ biholomorphic to  an open neighborhood $\Exp(N)$ of $\lam$, and such that    $$\Exp(\Sigma^{i-n}_k(f)\cap N)=\cV^i_k(U)\cap\Exp(N).$$ 
\end{enumerate}
\end{thrm}

To prove this theorem, we need some preliminary remarks.

We recall next a few topological facts. Denote by $D^b_c(X)$ the bounded derived category of constructible sheaves on $X$, and by $Perv(X)$ the abelian category of perverse sheaves on $X$.  Let $j:W\ra X$ be the open embedding and $i:D\ra X$ the closed embedding of the complement of $W$. For a  rank 1 local system $L$ on $W$, the shifted complex $L[n]$ is perverse on $W$. The complex $Rj_*(L[n])$ is also perverse.


\begin{lemma}\label{lemE}
$$\cV^{i}_k(U)=\{ \lam\in (\bC^*)^r\mid \dim H^{i-n}(u^{-1}Rj_*(L_\lam[n]))\ge k \}.$$
\end{lemma}
\begin{proof}
Note that $u^{-1}Rj_*(L_\lam)$ is the stalk of $Rj_*(L_\lam)$ at $O$. Hence $H^i(u^{-1}Rj_*(L_\lam))$ is the stalk of $R^ij_*(L_\lam)$. This is same as the global sections of $R^ij_*(L_\lam)$ in a small ball $B$ around $O$. By the definition of the higher direct image functors on sheaves, $$\Gamma(B,R^ij_*(L_\lam))=H^i(U, L_\lam),$$
and the claim follows.
\end{proof}

Denote by $DR_X$ the de Rham functor. Recall that, by the Riemann-Hilbert correspondence, 
$$DR_X: D^b_{rh}(\cD_X)\ra D^b_c(X)$$ is an equivalence  of triangulated categories that restricts to an equivalence of abelian categories
$$
DR_X:Mod_{rh}(\cD)\ra Perv(X).
$$
Here $Mod_{rh}(\cD)$ is the abelian category of regular holonomic left $\cD_X$-modules, and $D^b_{rh}(\cD_X)$ is the bounded derived category of complexes of $\cD_X$-modules with regular holonomic cohomology sheaves.

To simplify notation, let $$
\cD_X[s]f^{s}|_{s=\al}\mathop{=}^{\text{def}}\cD_X[s]f^{s}\otimes_{\bC[s]}\bC[s]/(s-\al).
$$
For an $r$-tuple of numbers $\al=(\al_1,\ldots,\al_r)$ and a number $k$, we use the notation
$$
\al-k=(\al_1-k,\ldots,\al_r-k).
$$

\begin{prop}\label{propED} Let $\al\in\bC^r$, $\lam=\Exp(\al)\in(\bC^*)^r$. Let $k\gg 0$ be a very big integer. Then $\cD_X[s]f^s|_{s=\al-k}$ is a regular holonomic $\cD_X$-module and $$DR_X(\cD_X[s]f^s|_{s=\al-k})=Rj_*(L_\lam[n]),$$
where 
$
\al-k=(\al_1-k,\ldots, \al_r-k)\in\bC^r.
$
\end{prop}

This result is known to $\cD$-module experts and follows from the following two Lemmas. For $\al\in\bC^r$, we denote by 
$
\cD_X f^\al
$
the left $\cD_X$-submodule of $\cO_X[f^{-1}]\cdot f^\al$ generated by the symbol $f^\al=\prod_{i=1}^rf_i^{\al_i}$ with the expected behavior under differential operators. There is a natural $\cD_X$-module morphism
$$
\cD_X[s]f^s|_{s=\al}\ra \cD_X f^{\al}
$$
obtained by specialization. This morphism is always surjective, but it is not necessarily an isomorphism. The following is a straight-forward consequence of the existence of non-zero Bernstein-Sato ideals and one can find the proof in \cite[Proposition 3.6]{OTa}, for example:

\begin{lemma} Let $\al\in\bC^r$ and $k\gg 0$. Then the natural $\cD_X$-module morphism
$$
\cD_X[s]f^s|_{s=\al-k}\ra \cD_X f^{\al-k}
$$
is an isomorphism.
\end{lemma}

Consider now on $W=X\setminus f^{-1}(0)$ the $\cD_W$-module generated by $f^\al$. This is a regular holonomic $\cD_W$-module such that $DR_W(\cD_Wf^\al)=L_\lam[n]$. The $\cD$-module direct image $j_*(\cD_Wf^\al)$ is the same as the sheaf theoretic direct image,
$$
j_*(\cD_Wf^\al)=\bigcup_{k\in \nn} f^{-k}\cD_X  f^\al.
$$
An immediate consequence of the existence of a non-zero Bernstein-Sato polynomial for the section $f^\al$ is that 
$$
j_*(\cD_Wf^\al)= \cD_X f^{-k} f^\al=\cD_X f^{\al-k}
$$
for integers $k\gg 0$, and that this is a regular holonomic $\cD$-module; for more details see \cite[page 25]{be}. To summarize:

\begin{lemma} Let $\al\in\bC^r$, $\lam=\Exp (\al)$, and $k\gg 0$. Then $\cD_X f^{\al-k}$ is a regular holonomic $\cD_X$-module and 
$$
DR_X (\cD_X f^{\al-k})=Rj_*L_\lam[n].
$$
\end{lemma}

\begin{cor} Let  $\al\in\bC^r$ and $\lam=\Exp(\al)\in(\bC^*)^r$. Suppose that $Re(\al_i)\ll 0$ for all $1\le i\le r$. Then
$$
 \al\in\Sigma^{i-n}_k(f) \iff \lam \in \cV^i_k(U).
$$
\end{cor}
\begin{proof} By the Riemann-Hilbert correspondence, there is an isomorphism of functors
$$
DR_O\circ u^\bigstar \xrightarrow{\sim} u^{-1}\circ DR_X : D^b_{rh}(\cD_X)\ra D^b_c(O).
$$
We can omit $DR_O$, since the de Rham functor on a point is trivial. The claim then follows from Lemma \ref{lemE} and Proposition \ref{propED}.
\end{proof}

\begin{rmk}
The condition $Re(\al_i)\ll 0$ in the above corollary can be made uniform by Sabbah's result \cite{sab} on Bernstein-Sato ideals (see \cite[Proposition 3.2]{OTa}). More precisely, there exists a large integer $N$ only depending on $f$, such that the above proposition still holds if the condition $Re(\al_i)\ll 0$ is replaced by $Re(\al_i)\leq-N$. 
\end{rmk}

\noindent
{\it Proof of  Part 2 of Theorem \ref{thrmS}.} This follows immediately from the Corollary. \hfill $\square$

\medskip
\noindent
{\it Proof of  Part 1 of Theorem \ref{thrmS}.} Let $\sigma$ be a field automorphism of $\bC$, and write $\sigma(\bC)$ for $\bC$ viewed as an algebra over itself via $\sigma$.  For a $\bC$-algebra $A$, we denote by $\sigma(A)$ the $(\sigma(\bC),A)$-bialgebra $\sigma(\bC)\otimes_\bC A$. For an $A$-module $M$, let $\sigma(M)$ be the $\sigma(A)$-module $\sigma(A)\otimes_A M$. For example, if $m$ is the ideal $(s-\al)=(s_1-\al_1,\ldots, s_r-\al_r)$ in $\bC[s_1,\ldots,s_r]$, then $\sigma(m)$ is the ideal $(s-\sigma(\al))$.
For two $A$-modules $M$ and $N$, $\sigma(M\otimes_A N)=\sigma(M)\otimes_{\sigma(A)}\sigma(N)$, $\sigma(Hom _A(M,N))=Hom_{\sigma(A)}(\sigma(M),\sigma(N))$, etc. 

From now on, let $\sigma\in Gal(\bC/\bar{\qq})$. For a fixed $\al\in\bC^r$, let 
$$
k(\al)=\dim_{\bC} H^i\left(u^{\bigstar}\left(\cD_X[s]f^s\otimes_{\bC[s]}\bC[s]/(s-\al)\right)\right).
$$
Then
\begin{align*}
k(\al)&=\dim_{\sigma(\bC)}\sigma\left(H^i\left(u^{\bigstar}\left(\cD_X[s]f^s\otimes_{\bC[s]}\bC[s]/(s-\al)\right)\right)\right)\\
&=\dim_\bC H^i\left(\sigma\left(u^{\bigstar}\left(\cD_X[s]f^s\otimes_{\bC[s]}\bC[s]/(s-\al)\right)\right)\right)\\
&=\dim_\bC H^i\left(\sigma\left(\bD_O\circ u^+\circ \bD_X\left(\cD_X[s]f^s\otimes_{\bC[s]}\bC[s]/(s-\al)\right)\right)\right).
\end{align*}

If $Y$ is a smooth algebraic variety defined over $\bar{\qq}$ , then $\sigma$ defines an auto-equivalence on $D^b(\cD_Y)$ since $\sigma(\cO_Y)=\cO_Y$ and $\sigma(\cD_Y)=\cD_Y$. If $\cM$ is in $D^b(\cD_Y)$, then
 $$\sigma(\bD_Y(\cM))=\bD_Y(\sigma(\cM)),$$
since
\begin{align*}
\sigma(R\mathcal{H}om_{\cD_Y}(\cM,\cD_Y)\mathop{\otimes}_{\cO_Y}\omega_Y^{\otimes -1})&=R\mathcal{H}om_{\sigma(\cD_Y)}(\sigma(\cM),\sigma(\cD_Y))\mathop{\otimes}_{\sigma(\cO_Y)}\sigma(\omega_Y)^{\otimes -1}\\
&=R\mathcal{H}om_{\cD_Y}(\sigma(\cM),\cD_Y)\mathop{\otimes}_{\cO_Y}\omega_Y^{\otimes -1}.
\end{align*}

If $g:Y\to Z$ is a $\bar{\qq}$-morphism of smooth $\bar{\qq}$-varieties and $\cM$ is in $D^b(\cD_Z)$, then
$$
\sigma(g^+(\cM))=g^+(\sigma(\cM)).
$$ 
Indeed, 
$$
\sigma(\cD_{Y\to Z}\mathop{\otimes}_{g^{-1}(\cD_Z)} g^{-1}(\cM)) =\sigma(\cD_{Y\to Z})\mathop{\otimes}_{g^{-1}(\sigma(\cD_Z))} g^{-1}(\sigma(\cM)),
$$
and  $$\sigma(\cD_{Y\to Z})=\sigma(\cO_Y\mathop{\otimes}_{g^{-1}(\cO_X)}g^{-1}\cD_Z)=\sigma(\cO_Y)\mathop{\otimes}_{g^{-1}(\sigma(\cO_X))}g^{-1}(\sigma(\cD_Z))=\cD_{Y\to Z}.$$
 
 Hence,
 \begin{align*}
k(\al)&=\dim_\bC H^i\left(\bD_O\circ u^+\circ \bD_X\circ \sigma\left(\cD_X[s]f^s\otimes_{\bC[s]}\bC[s]/(s-\al)\right)\right)\\
&=\dim_\bC H^i\left(u^\bigstar\circ \sigma\left(\cD_X[s]f^s\otimes_{\bC[s]}\bC[s]/(s-\al)\right)\right).
\end{align*}

Since $X$ is affine, we can work with global sections instead of sheaves. Let $D_{X}$ be the global sections of $\cD_X$.
Then
$$
k(\al)=\dim_\bC H^i\left(u^\bigstar\left(\sigma(D_X[s]f^s)\otimes_{\bC[s]}\bC[s]/(s-\sigma(\al))\right)\right).
$$

We now use the fact that $f_i$ are defined over $\bar{\qq}$. By \cite[Lemma 10]{BO}, $\Ann_{D_{X}[s]}(f^s)$ is obtained from the $D_{\bar{\qq}^n}[s]$-module
$
\Ann_{D_{\bar{\qq}^n}[s]}(f^s)
$
by applying $.\otimes_{\bar{\qq}}\cc$.  Hence 
$$
\sigma(D_X[s]f^s) = D_X[s]f^s.
$$

Hence
$$
k(\al)=\dim_\bC H^i\left(u^\bigstar\left(D_X[s]f^s\otimes_{\bC[s]}\bC[s]/(s-\sigma(\al))\right)\right).
$$
This proves that $\Sigma^i_k(f)$ is invariant under $Gal (\bC/\bar{\qq})$. \hfill $\square$

\section{Exponential maps}
Let $\Exp: \cc^r\to (\cc^*)^r$ be the exponential map defined by $\alpha\mapsto \exp(2\pi i\alpha)$. The following is proved in the Appendix of \cite{BW}. 

\begin{theorem}\label{thmbw}{\rm \cite{BW}}. 
Let $S$ be a closed irreducible subvariety of $\cc^n$ defined over $\bar{\qq}$, and let $T$ be a closed irreducible subvariety of $(\cc^*)^n$ defined over $\bar{\qq}$. Suppose that $\dim S=\dim T$ and $\Exp(S)\subset T$. Then $T$ is a torsion translate of a subtorus. 
\end{theorem}

Here we need a stronger version of the above theorem. Before stating the new result, we recall a powerful theorem about the arithmetic of affine tori. 

\begin{theorem}\label{dense}\cite{mm}
Let $Z$ be an irreducible subvariety of $(\cc^*)^n$ defined over $\bar{\qq}$. Suppose the set of torsion points on $Z$ is Zariski dense. Then $Z$ is a torsion translate of a subtorus. \end{theorem}

The following, like Theorem \ref{thmbw}, is generalization of the classical Gelfond-Schneider theorem, which says that if $\alpha$ and $e^{2\pi i\alpha}$ are both algebraic numbers, then $\alpha\in \qq$. The difference with respect to Theorem \ref{thmbw} is that one drops the requirement that  $S$ must be Zariski closed. This has to be compensated by invariance under a Galois action.

\begin{theorem}\label{goodexpo}
Let $S\subset \cc^r$ be a subset  which is preserved under the obvious action by $Gal(\cc/\bar{\qq})$. Let $T\subset (\cc^*)^r$ be a closed subvariety (not necessarily irreducible) defined over $\bar{\qq}$. Suppose that for any point $t\in T$ there is an open ball $N\subset \cc^r$ such that,
\begin{enumerate}
\item $t\in \Exp(N)$;
\item $\Exp$ induces an isomorphism $N\cong \Exp(N)$; and
\item $\Exp(S\cap N)=T\cap \Exp(N)$.
\end{enumerate}
Then each irreducible component of $T$ is a torsion translated subtorus. 
\end{theorem}
\begin{proof} Let $T'$ be an irreducible component of $T$. We first prove the case when $T'=\{q\}$ is a closed point. Since $T$ is a variety defined over $\bar{\qq}$, $q$ must have all the coordinates in $\bar{\qq}$. By the assumption in the theorem, there is a point $p\in S\cap N$ such that $\Exp(p)=q$. Thanks to the Gelfond-Schneider theorem, it suffices to show that all the coordinates of $p$ are in $\bar{\qq}$. 

Let $Y\subset \cc^r$ be the smallest closed subvariety defined over $\bar{\qq}$ which contains $p$. By definition, $Y$ is irreducible.
\begin{lemma}
Suppose $p=(p_1, \ldots, p_r)\in \cc^r$. Let $R=\bar{\qq}[p_1, \ldots, p_r]\subset\cc$ be the subring of $\cc$ generated by the coordinates of $p$. Then $Y_{\bar{\qq}}\cong \spec(R)$, where $Y_{\bar{\qq}}$ is the underlying $\bar{\qq}$-variety of $Y$. 
\end{lemma}
\begin{proof}
Denote the coordinates of $\cc^r$ by $(z_1, \ldots, z_r)$. Let the coordinate ring of $Y_{\bar{\qq}}$ be $R'$. Define a ring map $\phi: R'\to \cc$ by $g\mapsto g(p)$. Since $R'$ is a quotient ring of $\bar{\qq}[z_1, \ldots, z_r]$, $\phi(R')\subset R$. Conversely, for each $1\leq i\leq r$, the coordinate function $z_i$ defines a regular function on $Y_{\bar{\qq}}$. Since $z_i(p)=p_i$, $p_i\in \phi(R')$ for all $1\leq i\leq r$. Thus $R\subset \phi(R')$, and hence $R=\phi(R')$. 

Since $Y_{\bar{\qq}}$ is the smallest $\bar{\qq}$-subvariety of $\bar{\qq}^n$ whose complexification contains $p$, the kernel of $\phi$ is zero. \end{proof}

Suppose at least one of the coordinates of $p$ is transcendental. Then $Y$ has dimension at least one. 

\begin{lemma}
For a very general point $p'\in Y$, there exists $\sigma\in Gal(\cc/\bar{\qq})$ such that $\sigma(p)=p'$. 
\end{lemma}
\begin{proof}
By the preceding lemma, a point $p'\in Y$ corresponds to a ring map $p': R\to \cc$. Given any nonzero $g\in R$, we can consider it as a function on $Y$ via $Y_{\bar{\qq}}\cong \spec(R)$. Hence we can define the zero locus $V(g)\subset Y$, which is a proper subvariety of $Y$. 

Notice that $R$ has countably many elements. Thus, we can define a closed point $p'\in X$ to be very general if it is not contained in any $V(g)$ with $0\neq g\in R$. Clearly such a very general point $p'$ defines an injective ring map $p': R\to \cc$. Denote the quotient field of $R$ by $K$, then a very general $p'$ induces a ring map $p': K\to \cc$. By the existence of transcendental basis (which assumes axiom of choice), such a ring map can be extended to an automorphism $\sigma$ of $\cc$. Now it is straightforward to check that $\sigma(p)=p'$. 
\end{proof}

By the above Lemma, a very general point $p'\in Y\cap N$ is contained in $S$. Since $\Exp(S\cap N)=T\cap \Exp(N)$, a very general point in $\Exp(Y\cap N)$ is contained in $T$. Since $T$ is analytically closed and $\Exp$ is a covering map, $\Exp(Y\cap N)\subset T$. Therefore $T$ contains a component of dimension greater than zero through $q$. This is a contradiction to $T'=\{q\}$ being an isolated point. Hence $p$ is defined over $\bar{\qq}$.  This finishes the proof of the Theorem for the case when $T'$ is a point.

Suppose now that $T'$ is higher dimensional. Since it is an irreducible component of $T$ and $T$ is defined over $\bar{\qq}$, $T'$ is also defined over $\bar{\qq}$. Denote the union of all the other components of $T$ by $T''$, that is $T''=\overline{T\setminus T'}$.

Notice that for a torsion translated subtorus $V\subset (\cc^*)^r$, $V$ is defined over $\bar{\qq}$ and each connected component of $\Exp^{-1}(V)$ is an irreducible subvariety of $\cc^r$ defined over $\bar{\qq}$. When $T'$ is higher dimensional, we can use the above argument to show that whenever a torsion translated subtorus $V\subset (\cc^*)^n$ intersects $T'\setminus T''$ at discrete points, the points have to be torsion. 

Let $d$ be the dimension of $T'$. Let $\pi: (\cc^*)^r\to (\cc^*)^d$ be a general surjective map of algebraic groups such that $\pi(T')$ is Zariski dense in $(\cc^*)^d$. Let $U\subset (\cc^*)^d$ be a Zariski open set above which $\pi|_{T'\setminus T''}$ is surjective and \'etale. For any torsion point $\tau\in U$, $\pi^{-1}(\tau)$ is a torsion translated subtorus in $(\cc^*)^r$ and its intersection with $T'\setminus T''$ is transverse and non-empty. By the earlier argument, these intersection points are all torsion points in $(\cc^*)^r$. Therefore,
$$\pi(\{\text{torsion points in } T'\setminus T''\})\supset \{\text{torsion points in }U\}.$$
Since the dimension of $T'$ is $d$, the torsion points in $T'$ is also Zariski dense. By Theorem \ref{dense}, $T'$ is a torsion translate of a subtorus. This finishes the proof of Theorem \ref{goodexpo}.

\end{proof}

\section{The rest of the proofs}

\begin{theorem}\label{thrmMain} Let $U$ be the small ball complement of the germ of a holomorphic function $f:(\bC^n,0)\ra (\bC,0)$. Then each irreducible component of $\sV^i_k(U)$ is a torsion translate of subtorus of $\mb(U)$. 
\end{theorem}
\begin{proof}
The claim follows from Theorem \ref{Qbar2}, Theorem \ref{thrmS}, and Theorem \ref{goodexpo}. 
\end{proof}

To conclude the stronger Theorem \ref{thrmMain0}, we need the following result of J. Fern\'andez de Bobadilla \cite{FdB2}:

\begin{thrm}\label{thrmBo} Let $(\sX,0)$ be the germ of an analytic set in $(\bC^n,0)$, say given as the common zero locus of the germs of analytic functions $f_i:(\bC^n,0)\ra (\bC,0)$ $(i=1,\ldots, r)$. Let $U$ be the small ball complement of $(\sX,0)$. Then $U$ is homotopy equivalent with the Milnor fiber at $0$ of $f: (\bC^{n+r},0)\ra (\bC,0)$, where
$$
f(x_1,\ldots,x_n,y_1,\ldots, y_n)=\sum_{i=1}^ry_if_i(x_1,\ldots,x_n).
$$
Moreover, the geometric monodromy of $f$ is trivial.
\end{thrm}

\begin{proof}[Proof of Theorem \ref{thrmMain0}] Denote by $F_f$ the Milnor fiber at the origin for the germ $f$ defined as in Theorem \ref{thrmBo}. Let $U_f$ be the small ball complement of $f$. Since the geometric monodromy of $f$ is trivial, there is a diffeomorphism
$$
U_f\approx F_f\times \Delta^*,
$$
where $\Delta^*$ is a small punctured disc. Hence $U_f$ has the same homotopy type as $U\times \Delta^*$. By K\"unneth Theorem,
$$
H^1(U\times \Delta^*,\bC^*)=H^1(U,\bC^*)\times H^1(\Delta^*,\bC^*),
$$
and so $$\mb(U_f)=\mb(U\times \Delta^*)=\mb(U)\times \mb(\Delta^*)=\mb(U)\times \bC^*.$$ Hence a rank one local system on $U\times \Delta^*$ is uniquely determined by a pair $(L_1,L_2)\in \mb(U)\times \mb(\Delta^*)$.

The only non-empty $V^i_k(\Delta^*)$ with $k>0$ are $V^0_1(\Delta^*)=V^1_1(\Delta^*)=\{\bone_{\Delta^*}\}$. Hence, by K\"unneth Theorem again,
$$
H^i(U\times \Delta^*, (L,\bone_{\Delta^*}))=H^i(U,L)\oplus H^{i-1}(U,L)
$$
and $H^\ubul=0$ for the other rank one local systems on $U\times \Delta^*$. 

This implies that
$$
V^i_k(U_f)=V^i_k(U\times \Delta^*)=\mathop{\bigcup_{k_1,k_2\text{ such that }}}_{k_1+k_2=k} (V^i_{k_1}(U)\cap V^{i-1}_{k_2}(U))\times\{\bone_{\Delta^*}\}.
$$

Take now an irreducible component $T$ of $V^i_k(U)$. We will prove that $T$ is a torsion translated subtorus of $\mb(U)$. We can assume that $T$ is not included in $V^i_{k'}(U)$ for any $k'>k$. Define $l$ such that $T\subset V^{i-1}_l(U)$ but $T\not\subset V^{i-1}_{l'}(U)$ for any $l'>l$. Then $T$ is an irreducible component of $V^i_k(U)\cap V^{i-1}_l(U)$. Moreover, by the maximality of $k$ and $l$, $T$ remains an irreducible component of
$$
\mathop{\bigcup_{k',l'\text{ such that }}}_{k'+l'=k+l} (V^i_{k'}(U)\cap V^{i-1}_{l'}(U)).
$$
Hence $T\times\{\bone_{\Delta^*}\}$ is an irreducible component of $V^i_{k+l}(U_f)$. By Theorem \ref{thrmMain}, we know that each irreducible component of $V^i_{k+l}(U_f)$ is a torsion translated subtorus of $\mb(U_f)$. The claim follows now easily.
\end{proof}

\begin{proof}[Proof of Theorem \ref{corInv}] It is enough to prove that  $\lam\in \cV^i_k(X)$ if and only if $\lam^{-1}\in \cV^i_k(X)$. Let $\lam\in \cV^i_k(X)$. We can assume $\lam$ sits in a component $\rho\cdot T$ of $\cV^i_k(X)$, where $T$ is a subtorus in $\mb(X)$ and $\rho$ is torsion. Since $\cV^i_k(X)$ is an affine scheme defined over $\bQ$, it is invariant under the action of the Galois group of $\bar{\bQ}$ over $\bQ$. Let $\sigma\in Gal(\bar{\qq}/\qq)$ be the action such that $\sigma(\rho)=\rho^{-1}$. Then $\sigma (\rho T)$ must be a subset of $\cV^i_k(X)$. But $\sigma (\rho T)=\rho^{-1}\sigma (T)=\rho^{-1} T$, since any affine algebraic subtorus of $(\bC^*)^r$ is also defined over $\bQ$. If $\lam=\rho\cdot t$, for some $t\in T$, then $\lam^{-1}=\rho^{-1}\cdot t^{-1}\in \rho^{-1} T=\sigma (\rho T)\subset \cV^i_k(X)$.
\end{proof}

The above proof is the essentially the same as the proof of \cite[Corollary 4.9]{D-ad}.

\begin{proof}[Proof of Proposition \ref{propEig}] 
Let $f=\prod_{i=1}^r{f_i}$ be a decomposition of $f$ into irreducible holomorphic germs. Let $F=(f_1,\ldots,f_r)$. By \cite[Proposition 3.31]{B}, see also \cite[Example 3.32]{B}, one has
$$
\Supp _0(\psi_f\bC)=a^{-1}(\Supp _0(\psi_F\bC)),
$$
using Sabbah's specialization complex $\psi_F\bC$, where $a:\bC^*\ra(\bC^*)^r$ is the map $\lam\mapsto (\lam,\ldots,\lam)$. The left-hand side is the set of eigenvalues of the monodromy on the cohomology of the Milnor fiber of $f$ at $0$.

If $f$ is a reduced germ, \cite[Theorem 4]{B} implies that
$$
\Supp _0(\psi_F\bC)=\cV(U),
$$
which is what we claimed.

If $f$ is not a reduced germ, we use  \cite[Lemma 3.34]{B} and, again, \cite[Proposition 3.31]{B}, to obtain a map $b:(\bC^*)^r\ra \mb(U)$, such that $f^*=b\circ a: \bC^*\ra \mb (U)$ and
$$
\Supp _0(\psi_F\bC)=b^{-1}(\cV(U)).
$$
This finishes the proof.
\end{proof}

\begin{rmk} Note that, as pointed out by Liu-Maxim \cite{LM}, in all the statements in \cite{B} where the uniform support $Supp^{unif}_x(\psi_F\bC)$ of the Sabbah specialization complex appears, the {\it unif} should be dropped to conform to what is proven in \cite{B}. Indeed, the support $\Supp _x(\psi_F \bC)$ needs no uniformization.
\end{rmk}

\begin{proof}[Proof of Theorem \ref{thrmSab}]
Again, \cite[Lemma 3.34]{B} and \cite[Proposition 3.31]{B} reduce the statement to the case when $f=\prod_jf_i$ defines a reduced germ and the $f_i$ define the mutually distinct analytic branches. In this case, as above, \cite[Theorem 4]{B} gives that $
\Supp_0(\psi_F\bC)=\cV(U),
$
and claim then follows from Theorem \ref{thrmMain}.
\end{proof}

\section{Singular ambient space}
So far, we have always assumed that the ambient space of the analytic germ is $(\cc^n, 0)$, or equivalently a smooth germ. The goal of this section is to study the complement of an analytic subgerm of a possibly singular analytic germ and to obtain a generalization of our main result. We are not able to prove complete analogs of the algebraization of analytic germs in this setting. So, we have to assume the ambient germ and the subgerm are both algebraic. We will prove the following analog of Theorem \ref{thrmMain0}: 

\begin{theorem}\label{singmain}
Let $(\mathscr{X}, O)$ be a germ of a smooth complex algebraic variety, and let $(\mathscr{Y}, O)\subset (\scrX, O)$ be the germ of a possibly singular irreducible subvariety.  Let $f: (\mathscr{X}, O)\to (\cc, 0)$ be a reduced algebraic function germ. Let $f=\prod_{1\leq j\leq r}f_j$ be an irreducible decomposition of $f$ as an analytic germ on $(\scrX, O)$. Let $B\subset \scrX$ be a small ball centered at $O$. Assume that $U=(\mathscr{Y}\cap B)\setminus f^{-1}(0)$ is smooth. Let $F^*: \mb((\cc^*)^r)\to \mb(U)$ be the map of local systems induced by $F=(f_j)_{1\leq j\leq r}: U\to (\cc^*)^r$. Then each irreducible component of $(F^*)^{-1}(\sV^i_k(U))$ is a torsion translated subtorus in $\mb((\cc^*)^r)=(\cc^*)^r$. 
\end{theorem}
\begin{rmk}
Using the good neighborhood argument in \cite[Section II.B]{pr}, one can show that the homotopy type of $U$ is independent of the embedding $(\mathscr{Y}, O)\subset (\scrX, O)$. However, since we do not assume $(\mathscr{Y}, O)$ is factorial, the factorization $f=\prod_{1\leq j\leq r}f_j$ even as functions on $(\mathscr{Y}, O)$ does depend on the embedding $(\mathscr{Y}, O)\subset (\scrX, O)$. 
\end{rmk}

Given the germ of a variety $(\scrY, O)$, its link is defined to be $\link(\scrY)=\partial B\cap \scrY$ where $B$ is a small ball centered at $O$ and $\partial B$ is the boundary sphere of $B$. The homotopy type of $\link(\scrY)$ does not depend on the choice of $B$. An immediate consequence of the above theorem is the following. 
\begin{cor}
Let $(\scrY, O)$, $f=\prod_{1\leq j\leq r}f_j$ and $U$ be defined as in the above theorem. Suppose $(\scrY, O)$ has isolated singularity and suppose that $\link(\scrY)$ is simply connected. If the zero loci of $f_j$ in $(\scrY, O)$ are distinct irreducible divisors, then each irreducible component of $\sV^i_k(U)$ is a torsion translated subtorus. 
\end{cor}
\begin{proof}
It suffices to show that the map $F^*: \mb((\cc^*)^r)\to \mb(U)$ in the theorem is an isomorphism.

When $(\scrY, O)$ and $f$ satisfies the assumptions in the corollary,  the map $F_*: H_1(U,\zz)\to H_1((\cc^*)^r, \zz)$ induced by $F=(f_j)_{1\leq j\leq r}: U\to (\cc^*)^r$ is an isomorphism (see e.g. \cite{DL}). Therefore $F^*: \mb((\cc^*)^r)\to \mb(U)$ is an isomorphism. 
\end{proof}

\begin{proof}[Proof of Theorem \ref{singmain}]
First, we want to reduce to the case when all $f_i$ are algebraic functions. Since the ambient space is not necessarily smooth, our earlier approach in Section \ref{Def} does not apply any more. Here we use a different but similar argument. 

As in the proof of Proposition \ref{irreduciblepolynomial}, by multiplying each $f_j$ with some analytic unit, we can assume all $f_j$ are algebraic power series. Equivalently, there exists an \'etale map $\iota: \tilde{\scrX}\to \scrX$, whose image contains $O$, such that $f_j\circ \iota$ are algebraic functions on $\tilde{\scrX}$. Let $\tilde{\scrY}=\iota^{-1}\scrY$. Choose any $\tilde{O}\in \iota^{-1}(O)$. Analytically, the germ $(\scrY, O)$ together with functions $(f_j)_{1\leq j\leq r}$ is isomorphic to the germ $(\tilde{\scrY}, \tilde{O})$ together with functions $(f_j\circ \iota)_{1\leq j\leq r}$. Therefore, by replacing $(\scrY, O)\subset (\scrX, O)$ with $(\tilde{\scrY}, \tilde{O})\subset (\tilde{\scrX}, \tilde{O})$ and replacing $(f_j)_{1\leq j\leq r}$ with $(f_j\circ \iota)_{1\leq j\leq r}$, we can assume all $f_j$ are algebraic functions on $(\scrX, 0)$. 

Next, we need to reduce to the case when $\scrX$, $\scrY$ and $f_j$ are all defined over $\bar{\qq}$. We can simply take a common log resolution $\pi: \scrX'\to \scrX$ such that 
\begin{itemize}
\item every irreducible component of $\pi^{-1}(O)$ is of codimension 1;
\item $\pi^{-1}(\scrY)$ is smooth and birational to $\scrY$;
\item $(f\circ \pi)^{-1}(O)$ is a normal crossing divisor in $\scrX$;
\item $(f\circ \pi)^{-1}(O)\cap \pi^{-1}(\scrY)$ is a normal crossing divisor in $\scrY$;
\item the center of each blow-up is contained in the zero locus of $f$ or the singular locus of $\scrY$. In other words, $\pi$ is an isomorphism over $\scrX\setminus (\scrY_{\textrm{sing}}\cup f^{-1}(0))$. 
\end{itemize}
Then the same argument as in Proposition \ref{Qbarpolynomial} allows us to deform $\scrX$, $\scrY$ and $f_j$ simultaneously to varieties and functions defined over $\bar{\qq}$ without changing the small ball complement $U$ and the map $F=(f_j)_{1\leq j\leq r}: U\to (\cc^*)^r$, up to homotopy. 


Finally, we need to define the $\cD$-module cohomology jump loci and relate them with the local system jump loci as Theorem \ref{thrmS}. Since the construction and the proofs are essentially same as in Section \ref{dmod}, we shall only sketch the main steps.
 
Let $f=\prod_{1\leq j\leq r}f_j$. Consider the diagram

$$
\xymatrix{
& \scrY\setminus (f^{-1}(0)\cap \scrY)\ar[r]^{\ \ \ \ \ \ p} \ar[d]^q  &  \scrX\setminus f^{-1}(0) \ar[d]^j  \ar[r]^{\ \ \ \ F}& (\bC^*)^r\\
\{O\}\ar[r]^u & \scrY \ar[r]^i&   \scrX &
}
$$
The maps $i,j,p,q,u$ are the natural embeddings, $j$ and $q$ are open, and the square in the middle is cartesian. Everything is defined over $\bar{\qq}$.

Define the {\it $\cD$-module cohomology jump loci of $F$ in $(\scrY, O)$} by
$$
\Sigma^i_k(F, \scrY, O)=\left\{\al\in \cc^r\mid\dim_{\cc} H^i\left(u^\bigstar \circ i^+\left(\cD_\scrX[s]f^s\otimes_{\bC[s]}\bC[s]/(s-\al)\right)\right)\geq k\right\}.
$$
Here $\cD_\scrX$ is the ring of algebraic linear differential operators on $\scrX$, and $\cD_\scrX [s]f^s$ and $\bC[s]/(s-\al)$ are the same as in Section \ref{dmod}. Recall that the special pullback $i^+$ on $\cD$-modules
is the $\cD$-module counterpart of the special pullback $i^!$ on complexes of constructible sheaves.

Since everything in the diagram above is actually defined over $\bar{\bQ}$, it follows as in Section \ref{dmod} that  $\Sigma^i_k(F, \scrY, O)$ is invariant under the obvious action of $Gal(\cc/\bar{\qq})$ on $\cc^r$.

Next, we claim that when the real part of each $\alpha_i$ is sufficiently negative, $\alpha=(\alpha_i)_{1\leq i\leq r}\in \Sigma^i_k(F, \scrY, O)$ if and only if $\Exp(\alpha)\in (F^*)^{-1}(\sV^{i+2m-n}_k(U))$, where $m=\dim \scrY$ and $n=\dim \scrX$. By the Riemann-Hilbert correspondence, as before, the first of these conditions is equivalent with 
\begin{equation}\label{eqdd}
\dim _\bC H^i( u^{-1}\circ i^!\circ Rj_*(L_{\al}[n]))\ge k,
\end{equation}
where $L_\al$ is the local system on $\scrX\setminus f^{-1}(0)$ obtained as the pullback under $F$ of $\Exp(\al)\in M_B((\bC^*)^r)$. Since the square in the diagram above is cartesian, $$i^! \circ Rj_*=Rq_*\circ p^!.$$ Since $p$ is an embedding of smooth irreducible varieties, 
$$p^!(L_\al[n])=\bD\circ p^{-1}(L_{-\al}[n])=\bD ( p^{-1}L_{-\al} [n] ) = p^{-1}(L_{\al}) [2m-n],
$$
where $\bD$ is the Verdier dual. Hence (\ref{eqdd}) is equivalent to
$$
\dim _\bC H^i( u^{-1}\circ Rq_*(p^{-1}(L_{\al})[2m-n]))\ge k.
$$
This is equivalent with
$$
\dim _\bC H^{i+2m-n}( u^{-1}\circ Rq_*(p^{-1}(L_{\al})))\ge k.
$$
As before, the pullback $u^{-1}$ is the stalk at $O$, and hence the left-hand side computes the cohomology of $p^{-1}(L_{\al})$ on the small ball complement $U$. This proves the claim.

Thus, we proved the analog of Theorem \ref{thrmS}. Now, we can apply Theorem \ref{goodexpo} to finish the proof of the theorem. 
\end{proof}

\end{document}